\documentclass[USenglish,onecolumn]{article}
\usepackage[big]{dgruyter_NEW}
\usepackage{amsthm,amsmath,amssymb,amsfonts,natbib,bm,cases,multirow, booktabs, color}
\usepackage{graphicx,rotating,lscape,float}

\newtheorem{definition}{Definition}

\newtheorem{proposition}{Proposition}

\newtheorem{lemma}{Lemma}

\begin{document}

\articletype{Research Article{\hfill}}

\author[1]{Christian Genest}

\author[1]{Fr\'ed\'eric Ouimet}

\affil[1]{Department of Mathematics and Statistics, McGill University,
805, rue Sherbrooke ouest, Montr\'eal (Qu\'ebec) Canada H3A 0B9}

\title{\huge {A combinatorial proof of the Gaussian product inequality beyond the \texorpdfstring{$\mbox{MTP}_{\bf 2}$}{MTP2} case}}

\runningtitle{Combinatorial proof of the Gaussian product inequality}


\startpage{01}


\renewcommand{\lq}{\textquotedblleft}


\maketitle

\vspace{-3mm}
{\small
\noindent
\textbf{Abstract:}
A combinatorial proof of the Gaussian product inequality (GPI) is given under the assumption that each component of a centered Gaussian random vector $\bm{X} = (X_1, \ldots, X_d)$ of arbitrary length can be written as a linear combination, with coefficients of identical sign, of the components of a standard Gaussian random vector. This condition on $\bm{X}$ is shown to be strictly weaker than the assumption that the density of the random vector $(|X_1|, \ldots, |X_d|)$ is multivariate totally positive of order~$2$, abbreviated $\mbox{MTP}_2$, for which the GPI is already known to hold. Under this condition, the paper highlights a new link between the GPI and the monotonicity of a certain ratio of gamma functions.}




\section{Introduction\label{sec:1}}

The Gaussian product inequality (GPI) is a long-standing conjecture which states that for any centered Gaussian random vector $\bm{X} = (X_1, \dots, X_d)$ of dimension $d \in \mathbb{N} = \{ 1, 2, \ldots \}$ and every integer $m \in \mathbb{N}$, one~has
\begin{equation}
\label{eq:1}
\mbox{\rm E} \left(\prod_{i=1}^d X_i^{2 m}\right) \geq \prod_{i=1}^d \mbox{\rm E} \big( X_i^{2 m} \big).
\end{equation}
This inequality is known to imply the real polarization problem conjecture in functional analysis \citep{Malicet/etal:2016} and it is related to the so-called $U$-conjecture to the effect that if $P$ and $Q$ are two non-constant polynomials on $\mathbb{R}^d$ such that the random variables $P(\bm{X})$ and $Q(\bm{X})$ are independent, then there exist an orthogonal transformation $L$ on $\mathbb{R}^d$ and an integer $k \in \{ 1, \ldots, d - 1 \}$ such that $P \circ L$ is a function of $(X_1, \ldots, X_k)$ and $Q \circ L$ is a function of $(X_{k+1}, \ldots, X_d)$; see, e.g., \cite{Kagan/etal:1973, Malicet/etal:2016} and references therein.

Inequality~\eqref{eq:1} is well known to be true when $m = 1$; see, e.g., \citet{Frenkel:2008}. \citet{Karlin/Rinott:1981} also showed that it holds when the random vector $|\bm{X}| = (|X_1|, \ldots, |X_d|)$ has a multivariate totally positive density of order~$2$, denoted $\mbox{MTP}_2$. As stated in Remark~1.4 of their paper, the latter condition is verified, among others, in dimension $d = 2$ for all nonsingular Gaussian random pairs.

Interest in the problem has recently gained traction when \citet{Lan/etal:2020} established the inequality in dimension $d = 3$. Hope that the result might be true in general is also fueled by the fact, established by \citet{Wei:2014}, that for any reals $\alpha_1, \ldots, \alpha_d \in (-1/2, 0)$, one has
\begin{equation}
\label{eq:2}
\mbox{\rm E} \left( \prod_{i=1}^d |X_i|^{2 \alpha_i} \right) \geq \prod_{i=1}^d \mbox{\rm E} \big (|X_i|^{2 \alpha_i} \big).
\end{equation}
\citet{Li/Wei:2012} have further conjectured that the latter inequality holds for all reals $\alpha_1, \ldots, \alpha_d \in [0, \infty)$ and any centered Gaussian random vector~$\bm{X}$.

The purpose of this paper is to report a combinatorial proof of inequality \eqref{eq:2} in the special case where the reals $\alpha_1, \ldots, \alpha_d$ are nonnegative integers and when each of the components $X_1, \ldots, X_d$ of the centered Gaussian random vector $\bm{X}$ can be written as a linear combination, with coefficients of identical sign, of the components of a standard Gaussian random vector. A precise statement of this assumption is given as Condition~(III) in Section~\ref{sec:2}, and the proof of the main result, Proposition~\ref{prop:2}, appears in Section~\ref{sec:3}. It is then shown in Section~\ref{sec:4}, see Proposition~\ref{prop:3}, that this condition is strictly weaker than the assumption that the random vector $|\bm{X}|$ is $\mbox{MTP}_2$.

Coincidentally, shortly after the first version of the present paper was posted on arXiv, inequality~\eqref{eq:2} for all nonnegative integers $\alpha_1, \ldots, \alpha_d \in \mathbb{N}_0 = \{0, 1, \ldots\}$ was established under an even weaker assumption, stated as Condition (IV) in Section~\ref{sec:2}. The latter condition states that up to a change of sign, the components of the Gaussian random vector $\bm{X}$ are all nonnegatively correlated; see Lemma~2.3 of \citet{Russell/Sun:2022}. Therefore, the present paper's main contribution resides in the method of proof using a combinatorial argument closely related to the complete monotonicity of multinomial probabilities previously shown by \citet{Ouimet:2018} and \citet{Qi/etal:2020}.

All background material required to understand the contribution and put it in perspective is provided in Section~\ref{sec:2}. The statements and proofs of the paper's results are then presented in Sections~\ref{sec:3} and \ref{sec:4}. The article concludes with a brief discussion in Section~\ref{sec:5}. For completeness, a technical lemma due to \citet{Ouimet:2018}, which is used in the proof of Proposition~\ref{prop:2}, is included in the Appendix.

\section{Background\label{sec:2}}

First recall the definition of multivariate total positivity of order~2 ($\mbox{MTP}_2$) on a set $\mathcal{S} \subseteq \mathbb{R}^d$.

{\begin{definition}
\label{def:MTP2}
A density $f: \mathbb{R}^d \to [0, \infty)$ supported on $\mathcal{S}$ is said to be multivariate totally positive of order $2$, denoted $\mbox{MTP}_2$, if and only if, for all vectors $\bm{x} = (x_1, \ldots, x_d), \bm{y} = (y_1, \ldots, y_d) \in \mathcal{S}$, one has
$$
f(\bm{x} \vee \bm{y}) f(\bm{x} \wedge \bm{y}) \geq f(\bm{x}) f(\bm{y}),
$$
where $\bm{x} \vee \bm{y} = (\max(x_1,y_1), \ldots, \max(x_d, y_d))$ and $\bm{x} \wedge \bm{y} = (\min(x_1,y_1), \ldots, \min(x_d, y_d))$.
\end{definition}}

Densities in this class have many interesting properties, including the following result, which corresponds to Eq.~(1.7) of \citet{Karlin/Rinott:1981}.

{\begin{proposition}
\label{prop:1}
Let $\bm{Y}$ be an $\mbox{MTP}_2$ random vector on $\mathcal{S}$, and let $\varphi_1, \ldots, \varphi_r$ be a collection of nonnegative and (component-wise) non-decreasing functions on $\mathcal{S}$. Then
$$
\mbox{\rm E} \left\{ \prod_{i=1}^r \varphi_i(\bm{Y}) \right\} \geq \prod_{i=1}^r \mbox{\rm E} \left\{ \varphi_i(\bm{Y}) \right\}.
$$
\end{proposition}}

In particular, let $\bm{X} = (X_1, \ldots, X_d)$ be a $d$-variate Gaussian random vector with zero mean and nonsingular covariance matrix $\mbox{\rm var}(\bm{X})$. Suppose that the following condition holds.

\bigskip
{\begin{description}
\item [(I)]
The random vector $|\bm{X}| = (|X_1|, \ldots, |X_d|)$ belongs to the $\mbox{MTP}_2$ class on $[0, \infty)^d$.
\end{description}}

Under Condition~(I), the validity of the GPI conjecture \eqref{eq:2} for all reals $\alpha_1, \ldots, \alpha_d \in [0, \infty)$ follows from Proposition~\ref{prop:1} with $r = d$ and maps $\varphi_1, \ldots, \varphi_d$ defined, for every vector $\bm{y} = (y_1, \ldots, y_d) \in [0, \infty)^d$ and integer $i \in \{ 1, \ldots, d \}$, by
$$
\varphi_i (\bm{y}) = y_i^{2 \alpha_i}.
$$

When $\bm{X} = (X_1, \ldots,X_d)$ is a centered Gaussian random vector with covariance matrix $\mbox{\rm var} ({\bm X})$, Theorem~3.1 of~\citet{Karlin/Rinott:1981} finds an equivalence between Condition~(I) and the requirement that the off-diagonal elements of the inverse of $\mbox{\rm var} ({\bm X})$ are all nonpositive up to a change of sign for some of the components of $\bm{X}$. The latter condition can be stated more precisely as follows using the notion of signature matrix, which refers to a diagonal matrix whose diagonal elements are $\pm 1$.

{\begin{figure}[t!]
\centering
\includegraphics[width=150mm]{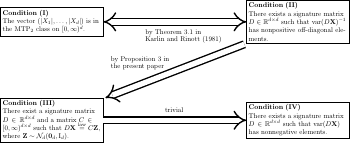}
\caption{Implications between Conditions~(I)--(IV) for a nonsingular centered Gaussian random vector $\bm{X}$, with references.}
\label{fig:1}
\end{figure}}

\bigskip
{\begin{description}
\item [(II)]
There exists a $d \times d$ signature matrix $D$ such that the covariance matrix $\mbox{\rm var}(D \bm{X})^{-1}$ only has nonpositive off-diagonal elements.
\end{description}}

Two other conditions of interest on the structure of the random vector $\bm{X}$ are as follows.
\bigskip
{\begin{description}
\item [(III)]
 There exist a $d \times d$ signature matrix $D$ and a $d \times d$ matrix $C$ with entries in $[0, \infty)$ such that the random vector $D \bm{X}$ has the same distribution as the random vector $C \bm{Z}$, where $\bm{Z} \sim \mathcal{N}_d(\bm{0}_d, \bm{I}_d)$ is a $d \times 1$ Gaussian random vector with zero mean vector $\bm{0}_d$ and identity covariance matrix $\bm{I}_d$.

\medskip
 \item [(IV)]
There exists a $d \times d$  signature matrix $D$ such that the covariance matrix $\mbox{\rm var}(D \bm{X})$ has only nonnegative elements.
\end{description}}

Recently, \citet{Russell/Sun:2022} used Condition~(IV) to show that, for all integers $d \in \mathbb{N}$, $n_1, \ldots, n_d \in \mathbb{N}_0$ and $k \in \{1, \ldots, d - 1\}$, and up to a change of sign for some of the components of~$\bm{X}$, one has
\begin{equation}
\label{eq:3}
\mbox{\rm E} \left( \prod_{i=1}^d X_i^{2 n_i} \right) \geq \mbox{\rm E }\left( \prod_{i=1}^k X_i^{2 n_i} \right) \, \mbox{\rm E} \left( \prod_{i=k+1}^d X_i^{2 n_i} \right).
\end{equation}
This result was further extended by \citet{Edelmann/Richards/Royen:2022} to the case where the random vector $(X_1^2, \ldots, X_d^2)$ has a multivariate gamma distribution in the sense of \citet{Krishnamoorthy/Parthasarathy:1951}. See also \cite{Bolviken:1982} for a use of Condition~(IV) in the context of the Gaussian correlation inequality (GCI) conjecture.

In the following section, it will be shown how Condition~(III) can be exploited to give a combinatorial proof of a weak form of inequality~\eqref{eq:3}. It will then be seen in Section~\ref{sec:4} that Condition~(II) implies Condition~(III), thereby proving the implications illustrated in Fig.~\ref{fig:1} between Conditions~(I)--(IV). That the implications Condition~(II) $\Rightarrow$ Condition~(III) and Condition~(III) $\Rightarrow$ Condition~(IV) are strict can be checked using, respectively, the covariance matrices
$$
\mbox{\rm var} ({\bm X}) = \begin{pmatrix}3/2 & 9/8 & 9/8 \\ 9/8 & 21/16 & 3/4 \\ 9/8 & 3/4 & 21/16\end{pmatrix} = \begin{pmatrix}1 & 1/2 & 1/2 \\ 1/2 & 1 & 1/4 \\ 1/2 & 1/4 & 1\end{pmatrix} \begin{pmatrix}1 & 1/2 & 1/2 \\ 1/2 & 1 & 1/4 \\ 1/2 & 1/4 & 1\end{pmatrix}
$$
and
$$
\mbox{\rm var} ({\bm X}) =
\begin{pmatrix} 1 & 0 & 0 & 1/2 & 1/2 \\ 0 & 1 & 3/4 & 0 & 1/2 \\ 0 & 3/4 & 1 & 1/2 & 0 \\ 1/2 & 0 & 1/2 & 1 & 0 \\ 1/2 & 1/2 & 0 & 0 & 1
\end{pmatrix} + \varepsilon \, \begin{pmatrix} 1 & 0 & 0 & 0 & 0 \\ 0 & 1 & 0 & 0 & 0 \\ 0 & 0 & 1 & 0 & 0 \\ 0 & 0 & 0 & 1 & 0 \\ 0 & 0 & 0 & 0 & 1
\end{pmatrix},
$$
for some appropriate $\varepsilon \in (0, \infty)$.

In the first example, the matrix $\mbox{\rm var} ({\bm X})$ is completely positive (meaning that it can be written as $C C^{\top}$ for some matrix $C$ with nonnegative entries) and positive definite by construction. Furthermore, the matrix $D \, \mbox{\rm var} ({\bm X})^{-1} D$ has at least one positive off-diagonal element for any of the eight possible choices of $3 \times 3$ signature matrix $D$. Another way to see this is to observe that if $A = \mbox{\rm var}(\bm X)^{-1}$, then the cyclic product $a_{12} a_{23} a_{31}$, which is invariant to $A\mapsto D A D$, is strictly positive in the above example, so that the off-diagonal elements of $\mbox{\rm var}(\bm X)^{-1}$ cannot all be nonpositive. This shows that (III) $\not\Rightarrow$~(II). This example was adapted from ideas communicated to the authors by Thomas Royen.

For the second example, when $\varepsilon = 0$, it is mentioned by~\citet{Maxfield/Minc:1962}, using a result from \citet{Hall:1958}, that the matrix is positive semidefinite and has only nonnegative elements but that it is not completely positive. Using the fact that the set of $5 \times 5$ completely positive matrices is closed, there exists $\varepsilon \in (0, \infty)$ small enough that the matrix $\mbox{\rm var} ({\bm X})$ is positive definite and has only nonnegative elements but is not completely positive. More generally, given that the elements of $\mbox{\rm var} ({\bm X})$ are all nonnegative, the matrix $D \, \mbox{\rm var} ({\bm X}) D$ is not completely positive for any of the $32$ possible choices of $5 \times 5$ signature matrix $D$, which shows that (IV)~$\not\Rightarrow$~(III). This idea was adapted from comments by~\citet{Stein:2011}.

\section{A combinatorial proof of the GPI conjecture\label{sec:3}}

The following result, which is this paper's main result, shows that the extended GPI conjecture of \citet{Li/Wei:2012} given in~\eqref{eq:2} holds true under Condition~(III) when the reals $\alpha_1, \ldots, \alpha_d$ are nonnegative integers. This result also follows from inequality~\eqref{eq:3}, due to \citet{Russell/Sun:2022}, but the argument below is completely different from the latter authors' derivation based on Condition~(IV).

{\begin{proposition}
\label{prop:2}
Let $\bm{X} = (X_1, \ldots, X_d)$ be a $d$-variate centered Gaussian random vector. Assume that there exist a $d \times d$ signature matrix $D$ and a $d \times d$ matrix $C$ with entries in $[0, \infty)$ such that the random vector $D \bm{X}$ has the same distribution as the random vector $C \bm{Z}$, where $\bm{Z}\sim \mathcal{N}_d(\bm{0}_d, \bm{I}_d)$ is a $d$-dimensional standard Gaussian random vector. Then, for all integers $n_1, \dots, n_d\in \mathbb{N}_0$,
$$
\mbox{\rm E} \left( \prod_{i=1}^d X_i^{2 n_i} \right) \geq \prod_{i=1}^d \mbox{\rm E} \left( X_i^{2 n_i} \right).
$$
\end{proposition}}

{\begin{proof}
In terms of $\bm{Z}$, the claimed inequality is equivalent to
\begin{equation}
\label{eq:4}
\mbox{\rm E}\left\{\prod_{i=1}^d \left(\sum_{j=1}^d c_{ij} Z_j\right)^{2 n_i}\right\} \geq \prod_{i=1}^d \mbox{\rm E}\left\{\left(\sum_{j=1}^d c_{ij} Z_j\right)^{2 n_i}\right\}.
\end{equation}

For each integer $j \in \{ 1, \ldots, d \}$, set $K_j = k_{1j} + \cdots + k_{dj}$ and $L_j = \ell_{1j}+ \cdots + \ell_{dj}$, where $k_{ij}$ and $\ell_{ij}$ are nonnegative integer-valued indices to be used in expressions \eqref{eq:5} and \eqref{eq:6} below.

By the multinomial formula, the left-hand side of inequality \eqref{eq:4} can be expanded as follows:
$$
\mbox{\rm E}\left\{\prod_{i=1}^d \sum_{\substack{\bm{k}_i\in \mathbb{N}_0^d : \\ k_{i1} + \dots + k_{id} = 2 n_i}} \binom{2 n_i}{k_{i1}, \dots, k_{id}} \prod_{j=1}^d c_{ij}^{k_{ij}} Z_j^{k_{ij}}\right\}.
$$
Calling on the linearity of expectations and the mutual independence of the components of the random vector $\bm{Z}$, one can then rewrite this expression as
\begin{multline*}
\mbox{\rm E}\left\{\sum_{\substack{\bm{k}_1\in \mathbb{N}_0^d : \\ k_{11} + \dots + k_{1d} = 2 n_1}} \dots \sum_{\substack{\bm{k}_d\in \mathbb{N}_0^d : \\ k_{d1} + \dots + k_{dd} = 2 n_d}} \prod_{i=1}^d \binom{2 n_i}{k_{i1}, \dots, k_{id}} \prod_{j=1}^d c_{ij}^{k_{ij}} Z_j^{k_{ij}}\right\} \\
= \sum_{\substack{\bm{k}_1\in \mathbb{N}_0^d : \\ k_{11} + \dots + k_{1d} = 2 n_1}} \dots \sum_{\substack{\bm{k}_d\in \mathbb{N}_0^d : \\ k_{d1} + \dots + k_{dd} = 2 n_d}} \left\{\prod_{j=1}^d \mbox{\rm E} \big (Z_j^{K_j} \big) \right\}  \prod_{i=1}^d \binom{2 n_i}{k_{i1}, \dots, k_{id}} \prod_{j=1}^d c_{ij}^{k_{ij}} .
\end{multline*}

\bigskip
Given that the coefficients $c_{ij}$ are all nonnegative by assumption, and exploiting the fact that, for every integer $j \in \{ 1, \ldots, d \}$ and $m\in \mathbb{N}_0$,
$$
\mbox{\rm E} \big( Z_j^{2m} \big) = \frac{(2m)!}{2^m m!} ,
$$
one can bound the left-hand side of inequality \eqref{eq:4} from below~by
\begin{multline}
\label{eq:5}
\sum_{\substack{\bm{\ell}_1\in \mathbb{N}_0^d : \\ 2 \ell_{11} + \dots + 2 \ell_{1d} = 2 n_1}} \dots \sum_{\substack{\bm{\ell}_d\in \mathbb{N}_0^d : \\ 2 \ell_{d1} + \dots + 2 \ell_{dd} = 2 n_d}} \left\{\prod_{j=1}^d \mbox{\rm E}\left(Z_j^{2 L_j}\right)\right\} \prod_{i=1}^d \binom{2 n_i}{2 \ell_{i1}, \dots, 2 \ell_{id}} \prod_{j=1}^d c_{ij}^{2 \ell_{ij}} \\
= \sum_{\substack{\bm{\ell}_1\in \mathbb{N}_0^d : \\ \ell_{11} + \dots + \ell_{1d} = n_1}} \dots \sum_{\substack{\bm{\ell}_d\in \mathbb{N}_0^d : \\ \ell_{d1} + \dots + \ell_{dd} = n_d}} \left\{\prod_{j=1}^d \frac{(2 L_j)!}{2^{L_j} L_j!}\right\} \prod_{i=1}^d \binom{2 n_i}{2 \ell_{i1}, \dots, 2 \ell_{id}} \prod_{j=1}^d c_{ij}^{2 \ell_{ij}}.
\end{multline}

\medskip
The right-hand side of \eqref{eq:4} can be expanded in a similar way. Upon using the fact that $\mbox{\rm E} (Y^{2m}) = (2m)! \sigma^{2m} / (2^m m!)$ for every integer $m \in \mathbb{N}_0$ when $Y \sim \mathcal{N}(0,\sigma^2)$, one finds

\begin{align}
\label{eq:6}
\prod_{i=1}^d \mbox{\rm E}\left\{\left(\sum_{j=1}^d c_{ij} Z_j\right)^{2 n_i}\right\} & = \prod_{i=1}^d \frac{(2 n_i)!}{2^{n_i} n_i!} \left(\sum_{j=1}^d c_{ij}^2\right)^{n_i}  = \prod_{i=1}^d \frac{(2 n_i)!}{2^{n_i} n_i!} \sum_{\substack{\bm{\ell}_i\in \mathbb{N}_0^d : \\ \ell_{i1} + \dots + \ell_{id} = n_i}} \binom{n_i}{\ell_{i1}, \dots, \ell_{id}} \prod_{j=1}^d c_{ij}^{2 \ell_{ij}} \nonumber \\
& = \sum_{\substack{\bm{\ell}_1\in \mathbb{N}_0^d : \\ \ell_{11} + \dots + \ell_{1d} = n_1}} \dots \sum_{\substack{\bm{\ell}_d\in \mathbb{N}_0^d : \\ \ell_{d1} + \dots + \ell_{dd} = n_d}} \prod_{i=1}^d \frac{(2 n_i)!}{2^{n_i} n_i!} \binom{n_i}{\ell_{i1}, \dots, \ell_{id}} \prod_{j=1}^d c_{ij}^{2 \ell_{ij}}.
\end{align}

Next, compare the coefficients of the corresponding powers $c_{ij}^{2 \ell_{ij}}$ in expressions \eqref{eq:5} and \eqref{eq:6}. In order to prove inequality \eqref{eq:4}, it suffices to show that, for all integer-valued vectors $\bm{\ell}_1, \dots, \bm{\ell}_d \in \mathbb{N}_0^d$ satisfying $\ell_{i1} + \dots + \ell_{id} = n_i$ for every integer $i \in \{1, \ldots, d \}$, one has
$$
\left\{\prod_{j=1}^d \frac{(2 L_j)!}{2^{L_j} L_j!}\right\} \prod_{i=1}^d \binom{2 n_i}{2 \ell_{i1}, \dots, 2 \ell_{id}} \geq \prod_{i=1}^d \frac{(2 n_i)!}{2^{n_i} n_i!} \binom{n_i}{\ell_{i1}, \dots, \ell_{id}}.
$$

Taking into account the fact that $2^{L_1 + \cdots + L_d} = 2^{n_1 + \cdots + n_d}$, and after cancelling some factorials, one finds that the above inequality reduces to
\begin{equation}
\label{eq:7}
\prod_{j=1}^d \frac{(2 L_j)!}{\prod_{i=1}^d (2 \ell_{ij})!} \geq  \prod_{j=1}^d \frac{L_j!}{\prod_{i=1}^d \ell_{ij}!}.
\end{equation}

Therefore, the proof is complete if one can establish inequality~\eqref{eq:7}. To this end, one can assume without loss of generality that the integers $L_1, \ldots, L_d$ are all non-zero; otherwise, inequality~\eqref{eq:7} reduces to a lower-dimensional case. For any given integers $L_1, \ldots, L_d \in \mathbb{N}$ and every integer $j \in \{1, \ldots, d \}$, define the function
$$
a \mapsto g_j (a) = \frac{\Gamma(a L_j + 1)}{\prod_{i=1}^d \Gamma(a \ell_{ij} + 1)},
$$
on the interval $(-1/L_j,\infty)$, where $\Gamma$ denotes Euler's gamma function.

To prove inequality \eqref{eq:7}, it thus suffices to show that, for every integer $j \in \{1, \ldots, d\}$, the map $a \mapsto \ln \{ g_j(a)\}$ is non-decreasing on the interval $[0, \infty)$. Direct computations yield, for every real $a \in [0,\infty)$,
\begin{align*}
\frac{d}{d a} \ln \{ g_j(a) \}
&= L_j \psi(a L_j + 1) - \sum_{i=1}^d \ell_{ij} \psi(a \ell_{ij} + 1), \\
\frac{d^2}{da^2} \ln \{ g_j(a) \}
&= L_j^2 \psi'(a L_j + 1) - \sum_{i=1}^d \ell_{ij}^2 \psi'(a \ell_{ij} + 1),
\end{align*}
where $\psi = (\ln \Gamma)'$ denotes the digamma function. Now call on the integral representation \cite[p.~260]{Abramovitz/Stegun:1964}
$$
\psi'(z) = \int_0^{\infty} \frac{t e^{-(z - 1)t}}{e^t - 1} \, d t,
$$
valid for every real $z \in (0, \infty)$, to write
\begin{align}
\label{eq:8}
\frac{d^2}{da^2} \ln \{ g_j (a) \}
& = \int_0^{\infty} \frac{(L_j t) e^{-a (L_j t)}}{e^t - 1} \, L_j d t - \sum_{i=1}^d \int_0^{\infty} \frac{(\ell_{ij} t) e^{-a (\ell_{ij} t)}}{e^t - 1} \, \ell_{ij} d t \nonumber \\
& = \int_0^{\infty} s e^{-a s} \left\{\frac{1}{e^{s / L_j} - 1} - \sum_{i=1}^d \frac{1}{(e^{s / L_j})^{L_j / \ell_{ij}} - 1}\right\} d s.
\end{align}

Given that $(\ell_{1j} + \cdots + \ell_{dj})/ L_j = 1$ by construction, the quantity within braces in Eq.~\eqref{eq:8} is always nonnegative by Lemma~1.4 of~\citet{Ouimet:2018}; this can be checked upon setting $y = e^{s / L_j}$ and $u_i = \ell_{ij} / L_j$ for every integer $i \in \{1, \ldots, d\}$ in that paper's notation. Alternatively, see p.~516 of~\citet{Qi/etal:2020}. Therefore,
\begin{equation}
\label{eq:9}
\forall_{a \in [0, \infty)} \quad \frac{d^2}{da^2} \ln \{ g_j(a) \}  \geq 0.
\end{equation}
In fact, the map $a\mapsto d^2 \ln \{ g_j(a) \} / d a^2$ is even completely monotonic. Moreover, given that
$$
\left.\frac{d}{d a} \ln  \{ g_j(a) \} \right|_{a = 0} = L_j \psi(1) - \sum_{i=1}^d \ell_{ij} \psi(1) = 0 \times \psi(1) = 0,
$$
one can deduce from inequality \eqref{eq:9} that
$$
\forall_{a\in [0,\infty)} \quad \frac{d}{d a} \ln  \{ g_j(a) \} \geq 0.
$$
Hence the map $a \mapsto \ln \{ g_j(a) \} $ is non-decreasing on $[0,\infty)$. This concludes the argument.
\end{proof}}

\section{Condition~(II) implies Condition~(III) \label{sec:4}}

This paper's second result, stated below, shows that Condition~(II) implies Condition~(III). In view of Fig.~\ref{fig:1}, one may then conclude that Condition~(I) also implies Conditions~(III) and (IV), and hence also Condition~(II) implies Condition~(IV). The implication Condition~(II) $\Rightarrow$ Condition~(IV) was already established in Theorem~2~(i) of \citet{Karlin/Rinott:1983}, and its strictness was mentioned on top of p.~427 of the same paper.

{\begin{proposition}
\label{prop:3}
Let $\Sigma$ be a symmetric positive definite matrix with Cholesky decomposition $\Sigma = C C^{\top}$. If the off-diagonal entries of $\Sigma^{-1}$ are all nonpositive, then the elements of $C$ are all nonnegative.
\end{proposition}}

{\begin{proof}
The proof is by induction on the dimension $d$ of $\Sigma$. The claim trivially holds when $d = 1$. Assume that it is verified for some integer $n \in \mathbb{N}$, and fix $d = n + 1$. Given the assumptions on $\Sigma$, one can write
$$
\Sigma^{-1} =
\begin{pmatrix}
a &\bm{v}^{\top} \\[1mm] \bm{v} &B
\end{pmatrix}
$$
in terms of a real $a \in (0, \infty)$, an $n \times 1$ vector $\bm{v}$ with nonpositive  components, and an $n \times n$ matrix $B$ with nonpositive off-diagonal entries.

Given that $\Sigma$ is symmetric positive definite by assumption, so is $\Sigma^{-1}$, and hence so are $B$ and $B^{-1}$. Moreover, the off-diagonal entries of $B = (B^{-1})^{-1}$ are nonpositive and hence by induction, the factor $L$ in the Cholesky decomposition $B^{-1} = L L^{\top}$ has nonnegative entries. Letting $w = a - \bm{v}^{\top} L L^{\top} \bm{v}$ denote the Schur complement, which is strictly positive, one has
$$
\begin{aligned}
\Sigma^{-1} =
\begin{pmatrix}
a & \bm{v}^{\top} \\[1mm] \bm{v} &(L L^{\top})^{-1}
\end{pmatrix}
=
\begin{pmatrix}
\sqrt{w} & \bm{v}^{\top} L \\ \bm{0}_n &(L^{\top})^{-1}
\end{pmatrix}
\begin{pmatrix}
\sqrt{w} &\bm{0}_n^{\top} \\ L^{\top} \bm{v} &L^{-1}
\end{pmatrix},
\end{aligned}
$$
where $\bm{0}_n$ is an $n \times 1$ vector of zeros. Accordingly,
$$
\Sigma =
\begin{pmatrix}
\sqrt{w} & \bm{0}_n^{\top} \\ L^{\top} \bm{v} &L^{-1}
\end{pmatrix}^{-1}
\begin{pmatrix}
\sqrt{w} & \bm{v}^{\top} L \\ \bm{0}_n &(L^{\top})^{-1}
\end{pmatrix}^{-1}
=
\begin{pmatrix}
1/\sqrt{w} & \bm{0}_n^{\top} \\  - L L^{\top} \bm{v}/\sqrt{w} &L
\end{pmatrix}
\begin{pmatrix}
1/\sqrt{w} & \bm{0}_n^{\top} \\ - L L^{\top} \bm{v}/\sqrt{w} &L
\end{pmatrix}^{\top} = C C^{\top}.
$$
Recall that $w$ is strictly positive and all the entries of $L$ and $-\bm{v}$ are nonnegative. Hence, all the elements of $C$ are nonnegative, and the argument is complete.
\end{proof}}

\section{Discussion\label{sec:5}}

This paper shows that the Gaussian product inequality~\eqref{eq:2} holds under Condition~(III) when the reals $\alpha_1, \ldots, \alpha_d$ are nonnegative integers. This assumption is further seen to be strictly weaker than Condition~(II). It thus follows from the implications in Fig.~\ref{fig:1} that when the reals $\alpha_1, \ldots, \alpha_d$ are nonnegative integers, inequality~\eqref{eq:2} holds more generally than under the $\mbox{MTP}_2$ condition of~\citet{Karlin/Rinott:1981}.

Shortly after the first draft of this article was posted on arXiv, extensions of Proposition~\ref{prop:2} were announced by~\citet{Russell/Sun:2022} and~\citet{Edelmann/Richards/Royen:2022}; see Lemma~2.3 and Theorem~2.1, respectively, in their manuscripts. Beyond priority claims, which are nugatory, the originality of the present work lies mainly in its method of proof and in the clarification it provides of the relationship between various assumptions made in the relevant literature, as summarized by Figure~\ref{fig:1}.

Beyond its intrinsic interest, the approach to the proof of the GPI presented herein, together with its link to the complete monotonicity of multinomial probabilities previously shown by \citet{Ouimet:2018} and \citet{Qi/etal:2020}, hints to a deep relationship between the $\mbox{MTP}_2$ class for the multivariate gamma distribution of~\citet{Krishnamoorthy/Parthasarathy:1951}, the range of admissible parameter values for their infinite divisibility, and the complete monotonicity of their Laplace transform; see the work of Royen on the GCI conjecture~\cite{Royen:2014, Royen:2015, Royen:2016, Royen:2022} and Theorems~1.2~and~1.3 of~\citet{Scott/Sokal:2014}.  These topics, and the proof or refutation of the GPI in its full generality, provide interesting avenues for future research.

\appendix

\section*{Appendix: Technical lemma\label{sec:technical.lemma}}

\renewcommand{\theequation}{A.\arabic{equation}}
\setcounter{equation}{0}

\renewcommand{\thelemma}{A.\arabic{lemma}}

The following result, used in the proof of Proposition~\ref{prop:2}, extends Lemma~1 of \citet{Alzer:2018} from the case $d = 1$ to an arbitrary integer $d \in \mathbb{N}$. It was already reported by \citet{Ouimet:2018}, see his Lemma~4.1, but its short statement and proof are included here to make the article more self-contained.

\begin{lemma}
\label{lemma:A1}
For every integer $d \in \mathbb{N}$, and real numbers $y \in (1, \infty)$ and $u_1, \ldots, u_{d+1} \in (0,1)$ such that $u_1 + \cdots + u_{d+1} = 1$, one has
\begin{equation}
\label{eq:A}
\frac{1}{y - 1} > \sum_{i=1}^{d+1} \frac{1}{y^{1/u_i} - 1}.
\end{equation}
\end{lemma}

\begin{proof}
The proof is by induction on the integer $d$. The case $d = 1$ is the statement of Lemma~1 of \citet{Alzer:2018}. Fix an integer $d \geq 2$ and assume that inequality~\eqref{eq:A} holds for every smaller integer. Fix any reals $y \in (1, \infty)$ and $u_1, \ldots, u_d \in (0, 1)$ such that $\|\bm{u}\|_1  = u_1 + \cdots + u_d < 1$. Write $u_{d+1} = 1 - \|\bm{u}\|_1 > 0$. Calling on Alzer's result, one has
$$
\frac{1}{y - 1} > \frac{1}{y^{1/\|\bm{u}\|_1} - 1} + \frac{1}{y^{1/(1 - \|\bm{u}\|_1)} - 1} .
$$
Therefore, the conclusion follows if one can show that
$$
\frac{1}{y^{1/\|\bm{u}\|_1} - 1} > \sum_{i=1}^d \frac{1}{y^{1/u_i} - 1} .
$$
Upon setting $z = y^{1/\|\bm{u}\|_1}$ and $v_i = u_i / \|\bm{u}\|_1$, one finds that the above inequality is equivalent to
$$
\frac{1}{z - 1} > \sum_{i=1}^d \frac{1}{z^{1/v_i} - 1} ,
$$
which is true by the induction assumption. Therefore, the argument is complete.
\end{proof}

\bigskip
\noindent
\textbf{Acknowledgments.} C.~Genest's research is funded in part by the Canada Research Chairs Program, 
the Trottier Institute for Science and Public Policy, and the Natural Sciences and Engineering Research Council of Canada. 
F.~Ouimet received postdoctoral fellowships from the Natural Sciences and Engineering Research Council of Canada and the Fond qu\'eb\'ecois de la recherche -- Nature et technologies (B3X supplement and B3XR). The authors are grateful to Donald Richards and Thomas Royen for their comments on earlier versions of this note. The authors also thank the referees for their comments.

\bigskip
\noindent
\textbf{Conflict of interest statement.} The authors state no conflict of interest.

\end{document}